\documentclass[12pt, two]{article}

\usepackage{CJK}
\usepackage{amsthm}
\usepackage{amsmath,amssymb,mathrsfs}
\usepackage{amsfonts}
\usepackage{yhmath}
\usepackage{graphics}
\usepackage{amssymb}
\usepackage{epsfig}
\usepackage{latexsym}
\usepackage{indentfirst}
\usepackage[top=1in, bottom=1in, left=1.25in, right=1.25in]{geometry}
\numberwithin{equation}{section}
\usepackage{CJK}

\newtheorem{theorem}{Theorem}[section]
\newtheorem{lemma}{Lemma}[section]
\newtheorem{proposition}{Proposition}[section]
\newtheorem{corollary}{Corollary}[section]
\newtheorem{remark}{Remark}[section]

\numberwithin{equation}{section} %¹«Ê½Ëæ½Ú±àºÅ

\title{
Liouville theorems for $F$-harmonic maps and their applications
\footnotetext[0]{2010 Mathematics Subject Classification. Primary: 35B53, 58Z05, 53C43. }
\footnotetext[0]{ ${}^*$Supported by NSFC grant No 10971029, and NSFC-NSF grant No 1081112053.}
\footnotetext[0]{{\em Key words and phrases. complete Riemannian manifold, F-harmonic map,
Liouville  theorem.}}
}
\author{
Yuxin Dong${}^*$,
  Hezi Lin and Guilin Yang}
\date{}
\begin{document}
%\begin{CJK*}{GBK}{kai}

\maketitle

\begin{abstract}
We prove several Liouville theorems for $F$-harmonic maps from some complete Riemannian manifolds by assuming some conditions on the Hessian of the distance function,  the degrees of $F(t)$ and the asymptotic behavior of the map at  infinity.  In particular, the results can be applied to $F$-harmonic maps from some  pinched manifolds, and can deduce a Bernstein type result for an entire minimal graph.
\end{abstract}

\section{Introduction}

In [Ar], M. Ara introduced the $F$-harmonic map and its associated stress-energy
tensor. The concept of $F$-harmonic maps unifies the concepts of harmonic
maps, $p$-harmonic maps,  minimal hypersurfaces, maximal spacelike hypersurfaces
and steady compressible flows, etc. It is known that the stress-energy tensor is a
useful tool for studying the energy behavior and vanishing results of related functional (cf. [DW]).

  Liouville type theorems for harmonic maps, $p$-harmonic maps and $F$-harmonic
maps were investigated by several authors (cf. [GRSB], [Ch], [Hi], [SY], [Se], [Ji],
[DW] and the references therein). Up to now, most Liouville results have been
established by assuming either the finiteness of the energy of the map or the smallness
of the whole image of the domain manifold under the map. In [Ji], Z.R. Jin
proved several interesting Liouville theorems for harmonic maps from complete
manifolds, whose assumptions concern the asymptotic behavior of the maps at infinity.
One special case of his results is that if $u: (R^m, g_0)\rightarrow (N^n, h)$ is a harmonic map,
and $u(x)\rightarrow  p_0 \in N^n$ as $|x| \rightarrow \infty$, then $u$ is a constant map.

   In this paper, we generalize Jin's method and results to $F$-harmonic maps.
The procedure consists of two steps. The first step is to use the $F$-stress energy
tensor, by choosing a suitable vector field, to deduce the lower energy rates of
the $F$-harmonic maps. The second step is to use the asymptotic assumption of
the maps at infinity to obtain the upper energy growth rates of the $F$-harmonic maps.
Under suitable conditions on $F$ and the Hessian of the distance functions of the
domain manifolds, one may show that these two growth rates are contradictory
unless the $F$-harmonic maps are constant maps. In this way, we establish some
Liouville results for $F$-harmonic maps with the asymptotic property at infinity from some complete manifolds (cf. $\S5$
for detailed statements).
Finally, in $\S 6$, we show that the asymptotic conditions on $F$-harmonic maps for
 Liouville theorems can be relaxed if the target manifold is more special.
 These Liouville theorems enable us to give an interesting
application for a global minimal graphic hypersurface $(x, u(x))$ in $R^{m+1}$  as follows:
If there exists a constant $c$ and
$\underset{R\rightarrow \infty}{\lim} \underset{|x|=R}{\max}\{\frac{(u-c)^2}{\sqrt{1+|du|^2}}\}=0$, then the graph is a horizontal hyperplane.

This paper is organized as follows. In  $\S 2$, we present some basic notions, some examples of $F$-harmonic maps and a useful integral formula associated with the $F$-stress energy tensor.  In  $\S 3$ and $\S 4$, under suitable conditions on the domain manifolds and the asymptotic condition on the maps, we derive the lower energy growth rates and the upper energy growth rates for $F$-harmonic maps respectively.  In  $\S 5$ and $\S 6$, we establish the main Liouville results and give some applications.

\section{Preliminaries}
Let $F: [0,\alpha)\rightarrow [0,\infty)$ be a $C^2$ function with $F(0)=0$ such that $F^{'}> 0$ on $(0, \alpha)$ for some $\alpha > 0$. For a smooth map $u: (M,g)\rightarrow (N,h)$ between Riemannian manifolds $(M,g)$ and $(N,h)$, we define the $F$-energy $E_F(u)$ by
\begin{equation}
E_F(u)=\int_MF(\frac{|du|^2}{2})dv_g= \int_M F(\frac{1}{2}h_{\alpha\beta}\frac{\partial u^\alpha}{\partial x_i}\frac{\partial u^\beta}{\partial x_j}g^{ij})dv_g.
\end{equation}
The map u is called $F$-harmonic if it is a critical point of the functional $E_F(u)$, that is,
\begin{equation}
\frac{d}{dt}E_F(u_t)=0
\end{equation}
for any compactly supported variation $u_t: M \rightarrow N$ $(-\epsilon < t < \epsilon)$ with $u_0 =u$.
the $F$-tension field $\tau_F(u)$ of $u$ is defined by
\begin{equation}
\tau_F(u)=F'(\frac{|du|^2}{2})\tau(u)+ u_*(grad(F'(\frac{|du|^2}{2}))).
\end{equation}
From [Ar], we know that $u$ is $F$-harmonic if and only if $\tau_F(u)=0$.

\textbf{Example 2.1}.  (cf. [Ar], [DW]) When $F(t)=2t$, $\frac{1}{p}(2t)^{p/2}$, $(1+2t)^\alpha$ ($\alpha>1$, dim$M$=2), and $e^{2t}$, the $F$-harmonic map becomes a harmonic map,
a $p$-harmonic map, an $\alpha$-harmonic map, and an exponentially harmonic map respectively.

\textbf{Example 2.2}. (cf. [Ya], [DW]) Let $M^m=(x, u(x)))\hookrightarrow R^{m+1}$ be a graph defined on $R^m$, where $u:R^m \rightarrow R$ be a smooth function. Then $M$ is a minimal graph if and only if $u: R^m \rightarrow R$ is a $F$-harmonic map with $F(t)=\sqrt{1+2t}-1$.

\textbf{Example 2.3}.  (cf. [Ot1,2]) Set $F(t)=\int_0^{2t} \rho (s)ds$ with $\rho: R^+\cup \{0\} \rightarrow R^+$. Under suitable conditions on $\rho$, the $F$-harmonic maps have a physical analogy as steady compressible flows on a Riemannian manifold.

From now on, we will assume that $F$ is defined on $[0, \infty)$, that is, $\alpha = \infty$.
Similar to [Ka] and [DW], we may define the upper degree $d_F$ and the lower degree $l_F$ of $F$  as follows:

\begin{equation*}
d_F=\underset{t \geq 0}{\sup}\frac{tF'(t)}{F(t)}
\end{equation*}
and
\begin{equation*}
l_F=\underset{t \geq 0}{\inf}\frac{tF'(t)}{F(t)}.
\end{equation*}
For example, if $F(t)=\frac{1}{p}(2t)^{p/2}$, then $d_F=l_F=p/2$; if $F(t)=\sqrt{1+2t}-1$, then $d_F=1$ and $d_l= \frac{1}{2}$.
In general, we have $l_F \leq d_F$.   From now on, we always assume that $d_F<+\infty$ and $m > \max\{2, 2d_F\}$.
The stress-energy tensor associated with the functional $E_F(u)$ is defined by ([Ar])
\begin{equation}
S_F(u)=F(\frac{|du|^2}{2})g- F'(\frac{|du|^2}{2})u^*(h).
\end{equation}
From [Ar], we know that if $u$ is $F$-harmonic, then $div S_F(u)=0$. Recall that for a two tensor field $T \in \Gamma(T^*M\otimes T^*M)$, it is divergence $div T \in \Gamma(T^*M)$ is defined by
\begin{equation*}
(div T)(X)=\underset{i}{\sum}(\nabla_{e_i}T)(e_i, X)
\end{equation*}
where $\{e_i\}$ is an orthonormal basis of $TM$. For a vector field $X$ on $M$, its dual one form $\theta_X$ is given by
\begin{equation*}
\theta_X(Y)=g(X,Y).
\end{equation*}
The covariant derivative of $\theta_X $ gives a 2-tensor field $\nabla\theta_X$:
\begin{equation}
(\nabla\theta_X)(Y,Z)=(\nabla_Z\theta_X)(Y)=g(\nabla_Z X,Y), \qquad \forall \ \ Y,Z\in TM.
\end{equation}
If $X=\nabla\psi$ is the gradient of some  smooth function $\psi$ on $M$, then $\theta_X=d\psi$ and $\nabla\theta_X=Hess(\psi)$.
\begin{lemma}
([Ba],[DW]) Let $T$ be a symmetric (0,2)-type tensor field and let $X$ be a vector field, then
\begin{eqnarray*}
div(i_X T)&=& (div T)(X) + \langle T, \nabla \theta_X \rangle\\
&=& (div T)(X) + \frac{1}{2}\langle T, L_Xg \rangle.
\end{eqnarray*}
\end{lemma}
\begin{proof}
Let $\{e_i\}$ be a local orthonormal frame field around a point $p$ such that $(\nabla e_i)_p=0$. Then
\begin{eqnarray}
\nonumber div(i_XT) &=& \underset{i=1}{\overset{m}{\sum}}(\nabla_{e_i}(i_XT))(e_i)\\
\nonumber &=& \underset{i=1}{\overset{m}{\sum}}(\nabla_{e_i}(T(X,e_i))- T(X, \nabla_{e_i}e_i))\\
\nonumber &=& (divT)(X) + \underset{i,j=1}{\overset{m}{\sum}}T(e_i, e_j)\langle \nabla_{e_i}X,e_j \rangle\\
&=& (divT)(X) + \langle T, \nabla \theta_X \rangle.
\end{eqnarray}
We also have
\begin{eqnarray}
\nonumber \langle L_Xg, T \rangle &=& \underset{i,j}{\sum}(L_Xg)(e_i,e_j)T(e_i, e_j)\\
\nonumber &=& \underset{i,j}{\sum}\{Xg(e_i,e_j)-g([X,e_i], e_j)-g(e_i,[X,e_j])\}T(e_i, e_j)\\
\nonumber &=& \underset{i,j}{\sum} 2g(\nabla_{e_i}X,e_j)T(e_i, e_j)\\
&=&  2\langle T, \nabla \theta_X \rangle.
\end{eqnarray}
Therefore (2.6) and (2.7) yield this lemma.
\end{proof}

Let $D$ be any bounded domain of $M$ with $C^1$ boundary.  By applying Lemma 2.1 to $S_F$ and using the divergence theorem, we immediately have the
following integral formula (see [Xi], [DW]):
\begin{equation}
\int_{\partial D}S_F(u)(X,\nu)ds_g=\int_{D}[\langle S_F(u),\frac{1}{2}L_Xg\rangle+(divS_F(u))(X)]dv_g
\end{equation}
where $\nu$ is the unit outward normal vector field along $\partial D$. In particular, if $u$ is a $F$-harmonic map, then by $divS_F(u)=0$,
we have
\begin{equation}
\int_{\partial D}S_{F}(u)(X,\nu)ds_g= \int_{D}\langle S_F(u), \frac{1}{2}L_Xg\rangle dv_g.
\end{equation}

\section{Lower energy growth rates  for $F$-harmonic maps}

Let $(M^m, g_0 )$ be a complete Riemannian manifold with a pole $x_0$. Denote by $r(x)$  the $g_0$-distance function relative to the pole $x_0$,
that is,  $r(x)=dist_{g_0}(x,x_0)$. Set $B(r)=\{x \in M^m: r(x) \leq r\}$. It is known that $\frac{\partial}{\partial r}$ is always an eigenvector of $Hess_{g_0}(r^2)$ associated to eigenvalue 2. Denote by $\lambda_{\max}$ (resp. $\lambda_{\min}$)  the maximum (resp. minimal ) eigenvalues of $Hess_{g_0}(r^2)-2dr\otimes dr$ at each point of $M\setminus\{x_0\}$.

From now on, we consider an $F$-harmonic map $u: (M^m, g)\rightarrow (N,h)$, where $g=f^2g_0$, $0< f \in C^{\infty}(M)$. Clearly the vector field
\begin{equation*}
\nu=f^{-1}\frac{\partial}{\partial r}
\end{equation*}
is an outer unit normal vector field along $\partial B(r)\subset (M, g)$.

Henceforth we will assume that $f$ satisfies either $\frac{\partial{\log f}}{\partial r} \geq 0$ or $\frac{\partial{\log f}}{\partial r} \leq 0$ on
$M\setminus\{x_0\}$. The further conditions for $f$ are as follows:\\
$(f_1)$ if $\frac{\partial{\log f}}{\partial r} \geq 0$ (resp. $ \frac{\partial{\log f}}{\partial r} \leq 0$), there is a constant $\sigma >0$
such that
\begin{equation*}
(m-2d_F)r \frac{\partial{\log f}}{\partial r} +\frac{(m-1)}{2}\lambda_{\min} + 1
- d_F\max\{2,\lambda_{\max}\} \geq \sigma
\end{equation*}
\begin{equation*}
(resp.
(m-2l_F)r \frac{\partial{\log f}}{\partial r} +\frac{(m-1)}{2}\lambda_{\min} + 1
- d_F\max\{2,\lambda_{\max}\}\geq \sigma).
\end{equation*}
$(f_2)$  there are constants $C>0$, $R_0 >0$   such that
\begin{equation*}
(\int_R^{\infty}\frac{dr}{\underset{\partial B(r)}{\int}f^{m-2}(x)ds_{g_0}})^{-1} \leq CR^{\sigma} \ \ for \ \ R>R_0,
\end{equation*}
where $\sigma$ is the constant in $(f_1)$.\\
$(f_3)$ there are constants $C>0$, $R_0>0$, such that
\begin{equation*}
 \underset{\partial B(R)}{\int}f^{m-2}(x)  ds_{g_0} \leq CR\log R \ \ for \ \ R>R_0.
 \end{equation*}

Now we take $X= r \frac{\partial}{\partial r} = \frac{1}{2}\nabla^0 r^2$ in (2.9), where $\nabla^0$ denotes the covariant derivative
determined by the metric $g_0$. By a direct computation, we have
\begin{eqnarray}
\nonumber \frac{1}{2}L_Xg &=& \frac{1}{2}L_{r \frac{\partial}{\partial r}}f^2g_0\\
\nonumber &=& rf\frac{\partial f}{\partial r}g_0 + f^2 \{\frac{1}{2}L_Xg_0\}\\
&=& r \frac{\partial{\log f}}{\partial r}g +  \frac{1}{2}f^2L_Xg_0,
\end{eqnarray}
and thus
\begin{eqnarray}
\nonumber g(S_F(u),\frac{1}{2}L_Xg ) &=& g(S_F(u), r \frac{\partial{\log f}}{\partial r}g +  \frac{1}{2}f^2L_Xg_0)\\
&=& r \frac{\partial{\log f}}{\partial r}g(S_F(u),g) + \frac{1}{2}f^2g(S_F(u), Hess_{g_0}(r^2)).
\end{eqnarray}
Using (2.4), we have
\begin{eqnarray}
(m-2l_F)F(\frac{|du|^2}{2}) \geq g(S_F(u),g) \geq  (m-2d_F)F(\frac{|du|^2}{2}).
\end{eqnarray}
Let $\{e_i\}_{i=1}^m$ be an orthonormal basis with respect to $g_0$ and $e_m=\frac{\partial}{\partial r}$. We may assume that $Hess_{g_0}(r^2)$ becomes a diagonal matrix w.r.t.$\{e_i\}$. Then $\{\tilde{e}_i= f^{-1}e_i\}_{i=1}^m$
is an orthonormal basis with respect to g.

\begin{eqnarray}
\nonumber f^2g(S_F(u), Hess_{g_0}(r^2)) &=& f^2\underset{i,j=1}{\overset{m}{\sum}}S_F(u)(\tilde{e}_i,\tilde{e}_j) Hess_{g_0}(r^2)(\tilde{e}_i,\tilde{e}_j)\\
\nonumber &=& f^2 \{ \underset{i=1}{\overset{m}{\sum}}F(\frac{|du|^2}{2}) Hess_{g_0}(r^2)(\tilde{e}_i,\tilde{e}_i)\\
\nonumber &&- \underset{i,j=1}{\overset{m}{\sum}}F'(\frac{|du|^2}{2})\langle du(\tilde{e}_i),du(\tilde{e}_j)\rangle Hess_{g_0}(r^2)(\tilde{e}_i,\tilde{e}_j)\}\\
&=& F(\frac{|du|^2}{2}) \triangle_{g_0}r^2 \\
\nonumber &&- F'(\frac{|du|^2}{2})\underset{i=1}{\overset{m}{\sum}}\langle du(\tilde{e}_i),du(\tilde{e}_i)\rangle Hess_{g_0}(r^2)(e_i, e_i)
\end{eqnarray}
Using the definition of  the upper degree, we have
\begin{eqnarray}
\nonumber f^2g(S_F(u), Hess_{g_0}(r^2)) &\geq& [(m-1)\lambda_{\min} + 2]F(\frac{|du|^2}{2})\\
\nonumber && -\max\{2,\lambda_{\max}\}F'(\frac{|du|^2}{2})\underset{i=1}{\overset{m}{\sum}}\langle du(\tilde{e}_i),du(\tilde{e}_i)\rangle\\
&\geq& [(m-1)\lambda_{\min} + 2- 2d_F\max\{2,\lambda_{\max}\} ]F(\frac{|du|^2}{2}).
\end{eqnarray}

In  following we only consider the case $\frac{\partial{\log f}}{\partial r} \geq 0$, because the argument for the other
case $\frac{\partial{\log f}}{\partial r} \leq 0$ is similar.

Combining (3.2), (3.3) and (3.5), we have
\begin{eqnarray}
\nonumber  g(S_F(u),\frac{1}{2}L_Xg ) &\geq&  [(m-2d_F)r \frac{\partial{\log f}}{\partial r} +\frac{(m-1)}{2}\lambda_{\min}
+ 1\\
&& - d_F\max\{2,\lambda_{\max}\}]F(\frac{|du|^2}{2}).
\end{eqnarray}
By the coarea formula and $|\nabla r|= f^{-1}$, we deduce that
\begin{eqnarray}
\nonumber \int_{\partial B(r)}S_F(u)(X, \nu)ds_g &\leq& \int_{\partial B(r)}F(\frac{|du|^2}{2})g(r \frac{\partial}{\partial r}, \nu)ds_g\\
\nonumber &=& r\int_{\partial B(r)}F(\frac{|du|^2}{2})f ds_g\\
\nonumber &=& r\frac{d}{dr}\int_0^r\{\frac{\int_{\partial B(t)}F(\frac{|du|^2}{2})ds_g}{|\nabla r|}\}dt\\
&=& r\frac{d}{dr}\int_{B(r)}F(\frac{|du|^2}{2})dv_g.
\end{eqnarray}
Hence, by (2.9), (3.6) and (3.7), we have
\begin{eqnarray}
\nonumber  r\frac{d}{dr}\int_{B(r)}F(\frac{|du|^2}{2})dv_g &\geq&  \int_{B(r)}[(m-2d_F)r \frac{\partial{\log f}}{\partial r} +\frac{(m-1)}{2}\lambda_{\min}+ 1\\
&& - d_F\max\{2,\lambda_{\max}\}]F(\frac{|du|^2}{2})dv_g.
\end{eqnarray}
By $(f_1)$,  there is  a constants $\sigma >0$  such that
\begin{equation}
(m-2d_F)r \frac{\partial{\log f}}{\partial r} +\frac{(m-1)}{2}\lambda_{\min} + 1
- d_F\max\{2,\lambda_{\max}\} \geq \sigma,\\
\end{equation}
thus
\begin{equation}
r\frac{d}{dr}\int_{B(r)}F(\frac{|du|^2}{2})dv_g \geq \sigma \int_{B(r)}F(\frac{|du|^2}{2})dv_g
\end{equation}
i.e.,
\begin{equation}
\frac{d}{dr}\frac{\int_{B(r)}F(\frac{|du|^2}{2})dv_g}{r^\sigma} \geq 0.
\end{equation}
Therefore
\begin{equation}
\frac{\int_{B(\rho_1)}F(\frac{|du|^2}{2})dv_g}{\rho_1^\sigma} \leq \frac{\int_{B(\rho_2)}F(\frac{|du|^2}{2})dv_g}{\rho_2^\sigma}
\end{equation}
for $0< \rho_1 \leq \rho_2$.

From the above discussion, we can get the following theorem.
\begin{proposition}
Let $u: (M^m,f^2g_0)\rightarrow (N,h)$ be a $C^2$ $F$-harmonic map. Suppose $f$ satisfies $(f_1)$. If $u$ is not a constant map,
then we have the following estimate
\begin{equation*}
\int_{B(R)}F(\frac{|du|^2}{2})dv_g \geq c(u)R^{\sigma} \ \ as \ R\rightarrow \infty
\end{equation*}
where $c(u)>0$ is a constant only depending on $u$.
\end{proposition}

Next, we show that if the $F$-harmonic has the unique continuation property(e.g., $F(t)=t$, $\sqrt{1+2t}-1$ in examples 2.2, etc.), the condition $(f_1)$
in Proposition 3.1 may be replaced by:

$(\widetilde{f}_1)$ The left hand sides of the inequalities in $(f_1)$ are nonnegative on the whole $M^m$ and there  exists an $R_0>0$ such that $(f_1)$ holds for $r(x)\geq R_0$.

Assuming $(\widetilde{f}_1)$, taking $X=r \frac{\partial}{\partial r}$ and applying Lemma 2.1 to $div(i_XS_F(u))$ on $B(R) \setminus B(R_0)$,
we get
\begin{eqnarray}
\nonumber && \int_{\partial (R)}S_F(u)(X,\nu)ds_g-\int_{\partial (R_0)}S_F(u)(X,\nu)ds_g\\
\nonumber &&= \int_{B(R) \setminus B(R_0)} g(S_F(u),\frac{1}{2}L_Xg )dv_g\\
 \nonumber  &&\geq  \int_{B(R) \setminus B(R_0)}[(m-2d_F)r \frac{\partial{\log f}}{\partial r} +\frac{(m-1)}{2}\lambda_{\min}
 + 1\\
  && \ \ \ \ - d_F\max\{2,\lambda_{\max}\}]F(\frac{|du|^2}{2})dv_g.
\end{eqnarray}
Set $H(R_0)=\int_{\partial (R_0)}S_F(u)(X,\nu)ds_g $, then by (3.7) and (3.13), we have
\begin{equation*}
R\frac{d}{dR}\int_{ B(R)}F(\frac{|du|^2}{2})ds_g -H(R_0) \geq \sigma \int_{B(R) \setminus B(R_0)}F(\frac{|du|^2}{2})dv_g,
\end{equation*}
and thus
\begin{equation}
R\frac{d}{dR}\{\int_{ B(R)\setminus B(R_0)}F(\frac{|du|^2}{2})dv_g +\frac{H(R_0)}{\sigma}\} \geq
\sigma \{\int_{ B(R)\setminus B(R_0)}F(\frac{|du|^2}{2})dv_g +\frac{H(R_0)}{\sigma}\}.
\end{equation}

To get the lower estimate of $F$-energy, we need the following lemma.
\begin{lemma}
Let $u$: $(M^m, f^2g_0)\rightarrow (N^n, h)$  be a  $C^2$ $F$-harmonic map with unique continuation property. Suppose $f$ satisfies $(\widetilde{f}_1)$. If $u$ is not a constant map, then the $F$-energy $E_F(u)$ must be infinite.
\end{lemma}

\begin{proof}
Using coarea formula, we have
\begin{eqnarray}
\nonumber \int_{M}F(\frac{|du|^2}{2})dv_g &=& \int_0^{+\infty}\frac{dr}{r}r\int_{\partial B(r)}F(\frac{|du|^2}{2})\frac{1}{|\nabla r|}ds_g\\
 &=& \int_0^{+\infty}\frac{dr}{r}r\int_{\partial B(r)}F(\frac{|du|^2}{2})fds_g
 \end{eqnarray}
 where $B(r)$ is the geodesic ball centered at $x_0$ with radius $r$.

 If $u$ has finite $F$-energy, i.e.  $\int_{M}F(\frac{|du|^2}{2})dv_g < +\infty$, then by $(3.15)$, we have a sequence $\{r_i\}$ such that
 \begin{equation}
 \underset{r_i \rightarrow +\infty}{\lim}r_i\int_{\partial B(r_i)}F(\frac{|du|^2}{2})fds_g=0.
 \end{equation}
 By $(\widetilde{f}_1)$, the inequality  $(3.9)$ holds for $r(x)\geq R_0$. Using (3.8), we have
 \begin{eqnarray}
\nonumber  r\frac{d}{dr}\int_{B(r)}F(\frac{|du|^2}{2})dv_g &\geq&  \sigma \int_{B(r)\setminus B(R_0)}F(\frac{|du|^2}{2})dv_g.
\end{eqnarray}
Let $r=r_i$ tend to infinity in the above inequality, using (3.16), we have $|du|^2=0$ on $M^m \setminus B(R_0)$, that is, $u$ is  constant  on $M^m \setminus B(R_0)$. By the unique continuation property, we deduce that $u$ is constant on $M^m$. This contradiction shows that the $F$-energy $E_F(u)$ must be infinite.
\end{proof}

By the above Lemma 3.1, we have
\begin{equation*}
\underset{R \rightarrow \infty}{\lim}\int_{ B(R)\setminus B(R_0)}F(\frac{|du|^2}{2})dv_g= +\infty,
\end{equation*}
and thus
\begin{equation*}
\int_{ B(R)\setminus B(R_0)}F(\frac{|du|^2}{2})dv_g +\frac{H(R_0)}{\sigma} >0
\end{equation*}
for sufficiently large $R$. Using (3.14), it follows that
\begin{equation*}
\frac{d}{dR}\{\frac{\int_{ B(R)\setminus B(R_0)}F(\frac{|du|^2}{2})dv_g +\frac{H(R_0)}{\sigma}}{R^{\sigma}}\}  >0
\end{equation*}
which implies that
\begin{equation*}
\int_{B(R)}F(\frac{|du|^2}{2})dv_g + \frac{H(R_0)}{\sigma}\geq c(u)R^{\sigma}
\end{equation*}
for sufficiently large $R$, where $c(u)>0$ is a constant only depending on $u$. Therefore we have the following proposition:
\begin{proposition}
Let $u: (M^m,f^2g_0)\rightarrow (N,h)$  be a $C^2$ $F$-harmonic map with unique continuation property. Suppose $f$ satisfies $(\widetilde{f}_1)$. If $u$ is not a constant map,
then we have the following estimate
\begin{equation*}
\int_{B(R)}F(\frac{|du|^2}{2})dv_g \geq c(u)R^{\sigma} \ \ as \ R\rightarrow \infty
\end{equation*}
where $c(u)>0$ is a constant only depending on $u$.
\end{proposition}

\section{Upper energy growth rates  for $F$-harmonic maps}

In order to get the Liouville type  property of $F$-harmonic maps, we need to estimate the upper $F$-energy of the  $F$-harmonic maps.

Set $$E_F^R(u)=\underset{B(R)}{\int}F(\frac{|du|^2}{2})dv_g.$$
Using a method similar to [Ji], we can derive the following theorem of an upper bound for the growth rate of $E_F^R(u)$ as $R\rightarrow \infty$.
\begin{proposition}
Let $u$: $(M^m, f^2g_0)\rightarrow (N^n, h)$ be a  $C^2$ $F$-harmonic map. Suppose that $f$ satisfies $(f_1)$ and $(f_2)$, and  the $F$-lower
degree $l_F>0$ and $F'(\frac{|du|^2}{2})< +\infty$. If $u(x)\rightarrow p_o \in N^n$ as $r(x)\rightarrow \infty$, then $u$ must be a constant map,
or there exists constants $R_0$, $c(u)$, and $\eta(R)\rightarrow 0$ as $R\rightarrow \infty$, such that
\begin{equation*}
E_F^R(u) \leq C(\frac{\eta(R)}{2l_F} + \frac{c(u)}{R^\sigma})R^\sigma \ \ \ for \ \ R \geq R_0.
\end{equation*}
 \end{proposition}

 \begin{proof}
 Suppose the $F$-harmonic map is not  constant,  then by  Proposition 3.1,  the $F$-energy of $u$ must be infinite. That is, $E_F^R(u)\rightarrow \infty$
  as $R\rightarrow \infty$.

  Choose a local coordinate neighbourhood $(U, \varphi)$ of $p_0$ in $N^n$, such that $\varphi(p_0)=0$, it is clear that we can choose the $U$
  in such a way that
  \begin{equation*}
  h=h_{\alpha\beta}(y)dy^\alpha \otimes dy^\beta, \qquad y \in U
  \end{equation*}
 satisfies
 \begin{equation*}
 (\frac{\partial h_{\alpha\beta(y)}}{\partial y^\gamma}y^\gamma +2h_{\alpha\beta}(y)) \geq (h_{\alpha\beta}(y)) \qquad on \ U
 \end{equation*}
 in the  matrices sense (that is, for two $n\times n$ matrices $A$, $B$, by $A\geq B$, we mean that $A-B$ is a positive semi-definite matrix).

 Now the assumption that $u(x)\rightarrow 0$ as $r(x)\rightarrow \infty$  implies that
 there is an $R_1$ such that for $r(x)>R_1$, $u(x) \in U$, and
 \begin{equation}
 (\frac{\partial h_{\alpha\beta(u)}}{\partial u^\gamma}y^\gamma +2h_{\alpha\beta}(u)) \geq (h_{\alpha\beta}(u)) \qquad for \ \ r(x)>R_1.
 \end{equation}

 For $w \in C^2_0(M^m\setminus B(R_1), \varphi(U))$, we consider the variation $u+tw: M^m \rightarrow N^n$ defined as follows:
 \begin{equation*}
(u+tw)(q)= \begin{cases}
u(q) & \text{if $q \in B(R_1)$,}\\
\varphi^{-1}[(\varphi(u)+ tw)(q)] & \text{if $q \in M^m \setminus B(R_1)$}
\end{cases}
\end{equation*}
for sufficient small $t$.
By the definition of $F$-harmonic maps,  we have
 $$\frac{d}{dt}|_{t=0}E_F(u+tw)=0$$
 that is,
 \begin{equation}
 \underset{M^m\setminus B(R_1)}{\int}g_0^{ij}F'(\frac{|du|^2}{2})[2h_{\alpha\beta}(u)\frac{\partial u^\alpha}{\partial x_i}\frac{\partial w^\beta}{\partial x_j} + \frac{\partial h_{\alpha\beta}(u)}{\partial y^\gamma}w^\gamma \frac{\partial u^\alpha}{\partial x_i}\frac{\partial u^\beta}{\partial x_j}]
 f^{m-2}(x)dv_{g_0}=0.
 \end{equation}
 Choose $w(x)=\phi(r(x))u(x)$ in (4.2) for $\phi(t)\in C^{\infty}_0(R_1, \infty)$, we have
 \begin{eqnarray}
\nonumber &{}&\underset{M^m\setminus B(R_1)}{\int}g_0^{ij}F'(\frac{|du|^2}{2})[2h_{\alpha\beta}(u)  + \frac{\partial h_{\alpha\beta}(u)}{\partial y^\gamma}u^\gamma] \frac{\partial u^\alpha}{\partial x_i}\frac{\partial u^\beta}{\partial x_j}\phi(r(x))f^{m-2}(x)dv_{g_0}\\
 &=& -2\underset{M^m\setminus B(R_1)}{\int}g_0^{ij}F'(\frac{|du|^2}{2})h_{\alpha\beta}(u)\frac{\partial u^\alpha}{\partial x_i}u^\beta
 \frac{\partial \phi(r(x))}{\partial x_j}f^{m-2}(x)dv_{g_0}.
 \end{eqnarray}
By a standard approximation argument, (4.3) holds for Lipschitz function $\phi$ with compact support.

For $0 < \epsilon \leq 1$, define
\begin{equation*}
\varphi_\epsilon(t)= \begin{cases}
1 & \text{$t \leq 1$}; \\
1+ \frac{1-t}{\epsilon}  & \text{$1< t< 1 + \epsilon$};\\
0  & \text{ $t \geq 1+\epsilon$}.
\end{cases}
\end{equation*}
In (4.3), choose the Lipschitz function $\phi(r(x))$  to be
 \begin{equation*}
 \phi(r(x))=\varphi_{\epsilon}(\frac{r(x)}{R})(1-{\varphi}_1(\frac{r(x)}{R_1})) , \ \ R >2R_1.
 \end{equation*}
 Set $\nu^i = g_0^{ij}\frac{\partial r}{\partial x_j}$,  and thus $\nu = \nu^i \frac{\partial}{\partial x_i}$ is the outer normal vector
 field along $\partial B(R)$.\\
 Let $\epsilon \rightarrow 0$, notice
 \begin{equation*}
 \frac{\partial\varphi_\epsilon(\frac{r(x)}{R}) }{\partial x_i} =- \frac{1}{R \epsilon}\frac{\partial r(x)}{\partial x_i} \ \ for \
 R< r(x) <R(1+ \epsilon)
 \end{equation*}
 and
 \begin{eqnarray*}
 &&\underset{\epsilon \rightarrow 0}{\lim}\frac{1}{R \epsilon}\underset{B(R(1+\epsilon))\setminus B(R)}{\int}F'(\frac{|du|^2}{2})h_{\alpha\beta}(u)\frac{\partial u^\alpha}{\partial x_i}\nu^iu^\beta
 f^{m-2}(x)dv_{g_0}\\
 &&=\underset{\partial B(R)}{\int}F'(\frac{|du|^2}{2})h_{\alpha\beta}(u)\frac{\partial u^\alpha}{\partial x_i}\nu^iu^\beta
 f^{m-2}(x)ds_{g_0},
 \end{eqnarray*}
  we get ($R_2=2R_1$)
 \begin{eqnarray}
 \nonumber &{}&\underset{B(R)\setminus B(R_2)}{\int}g_0^{ij}F'(\frac{|du|^2}{2})[2h_{\alpha\beta}(u)  + \frac{\partial h_{\alpha\beta}(u)}{\partial y^\gamma}u^\gamma] \frac{\partial u^\alpha}{\partial x_i}\frac{\partial u^\beta}{\partial x_j}f^{m-2}(x)dv_{g_0} + D(R_1)\\
 &=& 2\underset{\partial B(R)}{\int}F'(\frac{|du|^2}{2})h_{\alpha\beta}(u)\frac{\partial u^\alpha}{\partial x_i}\nu^iu^\beta
 f^{m-2}(x)ds_{g_0}
 \end{eqnarray}
 where
 \begin{eqnarray*}
 D(R_1)
  &=&\underset{B(R_2)\setminus B(R_1)}{\int}g_0^{ij}F'(\frac{|du|^2}{2})\{[(2h_{\alpha\beta}(u)  + \frac{\partial h_{\alpha\beta}(u)}{\partial y^\gamma}u^\gamma) \frac{\partial u^\alpha}{\partial x_i}\frac{\partial u^\beta}{\partial x_j}(1-\varphi_1(\frac{r(x)}{R_1}))]\\
  &&-2h_{\alpha\beta}(u)\frac{\partial u^\alpha}{\partial x_i}u^\beta\frac{\partial \varphi_1(\frac{r(x)}{R_1})}{\partial x_j}\}f^{m-2}(x)dv_{g_0}.
 \end{eqnarray*}
 Set
 \begin{equation*}
 Z(R)=\underset{B(R)\setminus B(R_2)}{\int}g_0^{ij}F'(\frac{|du|^2}{2})h_{\alpha\beta}(u)\frac{\partial u^\alpha}{\partial x_i}\frac{\partial u^\beta}{\partial x_j}f^{m-2}(x)dv_{g_0} + D(R_1) \ \ for \ \ R> R_2
 \end{equation*}
 then
  \begin{equation*}
 Z'(R)=\underset{\partial B(R)}{\int}g_0^{ij}F'(\frac{|du|^2}{2})h_{\alpha\beta}(u)\frac{\partial u^\alpha}{\partial x_i}\frac{\partial u^\beta}{\partial x_j}f^{m-2}(x)ds_{g_0}.
 \end{equation*}
 Notice that
 \begin{equation*}
 h_{\alpha\beta}(u)\frac{\partial u^{\alpha}}{\partial x_i}\nu^i u^{\beta}
 = \langle \frac{\partial u^{\alpha}}{\partial x_i} dx^i \otimes \frac{\partial}{\partial y^{\alpha}}, \frac{\partial r}{\partial x_j} u^{\beta} dx^j \otimes \frac{\partial}{\partial y^{\beta}}\rangle_{g_0 \otimes h},
 \end{equation*}
 and therefore
 \begin{eqnarray}
\nonumber && \int_{\partial B(R)}F^{'}(\frac{|du|^2}{2})h_{\alpha \beta}(u)\frac{\partial u^{\alpha}}{\partial x_i} \nu^i u^{\beta} f^{m-2}ds_{g_0}\\
\nonumber  &=& \int_{\partial B(R)}\langle \frac{\partial u^{\alpha}}{\partial x_i} dx^i \otimes \frac{\partial}{\partial y^{\alpha}}, \frac{\partial r}{\partial x_j} u^{\beta} dx^j \otimes \frac{\partial}{\partial y^{\beta}}\rangle_{g_0 \otimes h}F^{'}(\frac{|du|^2}{2})f^{m-2}ds_{g_0}\\
\nonumber  &\leq& \int_{\partial B(R)}|\frac{\partial u^{\alpha}}{\partial x_i} dx^i \otimes \frac{\partial}{\partial y^{\alpha}}|
 |\frac{\partial r}{\partial x_j} u^{\beta} dx^j \otimes \frac{\partial}{\partial y^{\beta}}|F^{'}(\frac{|du|^2}{2})f^{m-2}ds_{g_0}\\
\nonumber   &\leq& \sqrt{\int_{\partial B(R)}|\frac{\partial u^{\alpha}}{\partial x_i} dx^i \otimes \frac{\partial}{\partial y^{\alpha}}|^2_{g_0 \otimes h}|F^{'}(\frac{|du|^2}{2})f^{m-2}ds_{g_0}}\\
\nonumber   && \times \sqrt{\int_{\partial B(R)}|\frac{\partial r}{\partial x_j} u^{\beta} dx^j \otimes \frac{\partial}{\partial y^{\beta}}|^2_{g_0 \otimes h}|F^{'}(\frac{|du|^2}{2})f^{m-2}ds_{g_0}}
\end{eqnarray}

\begin{eqnarray}
\nonumber  &=& \sqrt{\int_{\partial B(R)}g_0^{ij}h_{\alpha \beta}(u)\frac{\partial u^{\alpha}}{\partial x_i}\frac{\partial u^{\beta}}{\partial x_j}F^{'}(\frac{|du|^2}{2})f^{m-2}ds_{g_0}}\\
\nonumber   && \times \sqrt{\int_{\partial B(R)}g_0^{ij}h_{\alpha \beta}(u)\frac{\partial r}{\partial x_i}\frac{\partial r}{\partial x_j}u^{\alpha}u^{\beta}F^{'}(\frac{|du|^2}{2})f^{m-2}ds_{g_0}}\\
\nonumber  &=& \sqrt{\int_{\partial B(R)}g_0^{ij}h_{\alpha \beta}(u)\frac{\partial u^{\alpha}}{\partial x_i}\frac{\partial u^{\beta}}{\partial x_j}F^{'}(\frac{|du|^2}{2})f^{m-2}ds_{g_0}}\\
  && \times \sqrt{\int_{\partial B(R)}h_{\alpha \beta}(u)u^{\alpha}u^{\beta}F^{'}(\frac{|du|^2}{2})f^{m-2}ds_{g_0}},
 \end{eqnarray}
where the last equality is because of
 \begin{equation*}
 g_0^{ij}\frac{\partial r}{\partial x_i}\frac{\partial r}{\partial x_j}=\langle \nabla^0 r, \nabla^0 r \rangle_{g_0}=1.
 \end{equation*}
 By the definition of $l_F$, we have
 \begin{equation}
 \underset{B(R)\setminus B(R_2)}{\int}g_0^{ij}F'(\frac{|du|^2}{2})h_{\alpha\beta}(u)\frac{\partial u^\alpha}{\partial x_i}\frac{\partial u^\beta}{\partial x_j}f^{m-2}(x)dv_{g_0} \geq 2l_F \underset{B(R)\setminus B(R_2)}{\int}F(\frac{|du|^2}{2})dv_g.
 \end{equation}
Since $l_F>0$ and $E_F^R \rightarrow \infty$ as $R\rightarrow \infty$, there is an $R_3 \geq R_2$, such that $Z(R)>0$ for $R \geq R_3$.
 Thus (4.1), (4.4) and (4.5) imply
\begin{equation*}
Z(R)^2 \leq CZ'(R)(\underset{\partial B(R)}{\int}F'(\frac{|du|^2}{2})h_{\alpha\beta}(u)u^\alpha u^\beta f^{m-2}(x)ds_{g_0})  \ \ for  \ \ R>R_3.
 \end{equation*}
 If we denote
\begin{equation}
 M(R)=\underset{\partial B(R)}{\int}F'(\frac{|du|^2}{2})h_{\alpha\beta}(u)u^\alpha u^\beta f^{m-2}(x)ds_{g_0},
 \end{equation}
 then for $R_4 \geq R \geq R_3$, it follows that
 \begin{equation*}
 \int_{R}^{R_4}(\frac{-1}{Z(r)})'dr \geq C \int_R^{R_4}\frac{1}{M(r)}dr.
 \end{equation*}
 Let $R_4 \rightarrow \infty$ and  notice that $Z(R)>0$, we have
 \begin{equation*}
 \frac{1}{Z(R)}\geq C \int_R^{\infty}\frac{1}{M(r)}dr.
 \end{equation*}
Thus
 \begin{equation*}
 Z(R) \leq C\frac{1}{\int_R^{\infty}\frac{1}{M(r)}dr} \ \ for \ R > R_3.
 \end{equation*}
 By  $F'(\frac{|du|^2}{2})< +\infty$ and the fact that $u(x)\rightarrow 0$ as $r(x)\rightarrow \infty$, we get
 \begin{equation*}
 M(R) \leq C \eta(R) \underset{\partial B(R)}{\int}f^{m-2}(x)ds_{g_0},
 \end{equation*}
 where $\eta(R)$  is chosen in such a way that\\
 (i) $\eta(R)$ is nonincreasing on $(R_3, \infty)$ and $\eta(R)\rightarrow 0$ as $R\rightarrow \infty$;\\
 (ii)  $\eta(R) \geq \underset{r(x)=R}{\max}\{h_{\alpha\beta}(u)u^\alpha u^\beta\}.$\\
 Then by $(f_2)$, we derive
 \begin{equation*}
\int_R^{\infty}\frac{1}{M(r)}dr \geq \frac{C}{\eta(R)}\int_R^{\infty}\frac{1}{\underset{\partial B(r)}{\int}f^{m-2}(x)ds_{g_0}}dr \geq \frac{C}{\eta(R)}R^{-\sigma}.
\end{equation*}
Thus
\begin{equation*}
Z(R) \leq C\eta(R)R^\sigma \ \  for \ R \geq R_3.
\end{equation*}
Therefore, using (4.6), we obtain
\begin{equation*}
E_F^R(u) \leq C(\frac{\eta(R)}{2l_F} + \frac{c(u)}{R^\sigma})R^\sigma.
\end{equation*}
 \end{proof}

\begin{remark}
 When the $F$-harmonic map $u$ has the unique continuation property, then by Lemma 3.1, the conclusion of Proposition 4.1 also holds for $u$ with the condition $(f_1)$ replaced by $(\widetilde{f}_1)$.
\end{remark}

\section{The main results and their proof}

Combining Proposition 3.1 and Proposition 4.1, we have the following Liouville type theorem.
\begin{theorem}
Let $u$: $(M^m, f^2g_0)\rightarrow (N^n, h)$ be a  $C^2$ $F$-harmonic map. Suppose that $f$ satisfies $(f_1)$ and $(f_2)$, and that  the $F$-lower degree
$l_F>0$ and $F'(\frac{|du|^2}{2})< +\infty$.
If $u(x)\rightarrow p_o \in N^n$ as $r(x)\rightarrow \infty$, then $u$ is a constant map.
 \end{theorem}

\begin{remark}
Assuming either $\underset{t\geq 0}{sup} F^{'}(t)< + \infty$ or $\underset{M}{sup} |du|^2<+ \infty$, we may deduce that $F^{'}(\frac{|du|^2}{2})< C$.
\end{remark}

By  Proposition 3.2 and  Remark 4.1, we have the following theorem for $F$-harmonic map with the unique continuation property, which includes the case of harmonic maps in Jin's paper ([Ji]).
\begin{theorem}
Let $u$: $(M^m, f^2g_0)\rightarrow (N^n, h)$ be a  $C^2$ $F$-harmonic map with the unique continuation property. Suppose that  $f$ satisfies $(\widetilde{f}_1)$ and $(f_2)$, and that the $F$-lower degree
$l_F>0$ and $F'(\frac{|du|^2}{2})< +\infty$.
If $u(x)\rightarrow p_o \in N^n$ as $r(x)\rightarrow \infty$, then $u$ is a constant map.
 \end{theorem}

\begin{corollary}
Let $u$: $(M^m, g_0)\rightarrow (N^n, h)$ be a  $C^2$ $F$-harmonic map. There are positive constants $C$, $\sigma$ and $R_0$ such that
\begin{equation}
\frac{(m-1)}{2}\lambda_{\min} + 1
- d_F\max\{2,\lambda_{\max}\} \geq \sigma,
\end{equation}
\begin{equation}
(\int_R^{\infty}\frac{dr}{vol(\partial B(r))})^{-1} \leq CR^{\sigma} \ \ for \ \ R>R_0.
\end{equation}
 The lower degree
$l_F>0$ and $F'(\frac{|du|^2}{2})< +\infty$.
If $u(x)\rightarrow p_o \in N^n$ as $r(x)\rightarrow \infty$, then $u$ is  a constant map.
 \end{corollary}

When applied the above results  to some concrete pinched manifolds, we need the following lemmas.
\begin{lemma}
Let $(M,g_0)$ be an $m$-dimensional complete Riemannian manifold with a pole $x_0$ and let $r(x)$ be the distance function relative to $x_0$. Assume
that there exist two positive functions $h_1(r)$ and $h_2(r)$ such that
\begin{equation*}
h_1(r)[g_0-dr\otimes dr] \leq Hess(r) \leq h_2(r)[g_0-dr\otimes dr]
\end{equation*}
in the sense of  quadratic forms, then
\begin{equation*}
\frac{(m-1)}{2}\lambda_{\min} + 1
- d_F\max\{2,\lambda_{\max}\} \geq (m-1)h_1(r)r + 1- 2d_F\max\{1,h_2(r)r\}.
\end{equation*}
\end{lemma}
\begin{proof}
Applying the Hessian operator to the composed function $r^2$, we have
\begin{equation*}
Hess(r^2)= 2r Hess(r)+ 2dr \otimes dr
\end{equation*}
which immediately yields the result.
\end{proof}

\begin{lemma}
(cf. [GW], [DW], [PRS])
Let ($M, g_0)$ be a complete Riemannian manifold with a pole  $x_0$ and let $r$ be the distance function relative to $x_0$.
Denote by $K_r$ the radial curvature of $M$.

(i) If $-\alpha^2\leq K_r \leq -\beta^2$ with $\alpha> 0$, $\beta> 0$, then
$$\beta \coth(\beta r)[g-dr \otimes dr] \leq Hess(r) \leq \alpha \coth(\alpha r)[g-dr \otimes dr].$$

(ii) If $-\frac{A}{(1+r^2)^{1+\epsilon}} \leq K_r \leq \frac{B}{(1+r^2)^{1+\epsilon}}$ with $\epsilon>0$, $A \geq 0$, $0\leq B < 2\epsilon$,
then
$$\frac{1-\frac{B}{2\epsilon}}{r} [g-dr \otimes dr] \leq Hess(r) \leq \frac{e^{\frac{A}{2\epsilon}}}{r}[g-dr \otimes dr]. $$

(iii) If $-\frac{a^2}{1+r^2} \leq K_r \leq \frac{b^2}{1+r^2}$ with $a\geq 0$, $b^2 \in[0,1/4]$, then
$$\frac{1+\sqrt{1-4b^2}}{2r} [g-dr \otimes dr] \leq Hess(r) \leq \frac{1+\sqrt{1+4a^2}}{2r}[g-dr \otimes dr].$$
\end{lemma}

By Theorem 5.1, Lemma 5.1 and Lemma 5.2, we can get the following theorem.
\begin{theorem}
Let $u$: $(M^m, f^2g_0)\rightarrow (N^n, h)$ be a  $C^2$ $F$-harmonic map. Suppose $f$ satisfies  $\frac{\partial \log f}{\partial r} \geq 0$ and
there exist a constant $\sigma>0$ such that the inequality in  $(f_2)$ holds. The $F$-lower
degree $l_F>0$ and $F'(\frac{|du|^2}{2})< +\infty$.
Suppose   $M^m$ is a complete Riemannian manifold with a pole and
 its radial curvature satisfies one of the following three conditions:

(i) $-\alpha^2\leq K_r \leq -\beta^2$ with  $\alpha> 0$, $\beta> 0$ and
$r\frac{\partial \log f}{\partial r}(m-2d_F) +1 + (m-1)\beta r\coth(\beta r)-2 d_F \alpha r\coth(\alpha r) \geq \sigma $;

(ii) $-\frac{A}{(1+r^2)^{1+\epsilon}} \leq K_r \leq \frac{B}{(1+r^2)^{1+\epsilon}}$ with $\epsilon>0$, $A \geq 0$, $0 \leq B < 2\epsilon$
and $r\frac{\partial \log f}{\partial r}(m-2d_F) +1 + (m-1)(1-\frac{B}{2\epsilon})-2 d_F e^{\frac{A}{2\epsilon}} \geq \sigma$;

(iii) $-\frac{a^2}{1+r^2} \leq K_r \leq \frac{b^2}{1+r^2}$ with $a\geq 0$, $b^2 \in[0,1/4]$ and
$r\frac{\partial \log f}{\partial r}(m-2d_F) +1 + (m-1)\frac{1+\sqrt{1-4b^2}}{2}- d_F (1+\sqrt{1+4a^2}) \geq \sigma $.

If $u(x)\rightarrow p_o \in N^n$ as $r(x)\rightarrow \infty$, then $u$ is a constant map.
 \end{theorem}

\begin{remark}
(i) If $(m- 1)\beta -2d_F \alpha \geq 0$, then we have
\begin{eqnarray*}
(m-1)\beta r \coth(\beta r) - 2d_F \alpha r \coth( \alpha r) &\geq& \beta r \coth(\beta r)[(m-1)-2d_F \frac{\alpha r \coth( \alpha r)}{\beta r \coth(\beta r)}]\\
&\geq& (m-1)-2d_F \frac{\alpha}{\beta}
\end{eqnarray*}
since $\beta r \coth(\beta r)>1$ for $r>0$, and $\frac{\alpha r \coth( \alpha r)}{\beta r \coth(\beta r)} <1$ for $0< \beta <\alpha$, and $\coth$
is a decreasing function. Thus the conclusion of the first case in Theorem 5.3 still holds if  $r\frac{\partial \log f}{\partial r}(m-2d_F) +m -2d_F \frac{\alpha}{\beta} \geq \sigma$.

(ii) If $f$ satisfies $\frac{\partial \log f}{\partial r} \leq 0$ and the condition $(f_2)$, then the  conclusion in Theorem 5.3 still holds provided that one replaces $m-2d_F$ by $m-2l_F$ in (i) (ii) (iii).

(iii) M. Kassi ([Ka]) proved a Liouville theorem for $F$-harmonic maps from various pinched manifolds, which has  finite $F$-energy and some restrictions on $d_F$.
\end{remark}

Taking $f=1$, we have the following corollary.
\begin{corollary}
Let $u$: $(M^m, g_0)\rightarrow (N^n, h)$ be a  $C^2$ $F$-harmonic map.  The $F$-lower
degree $l_F>0$ and $F'(\frac{|du|^2}{2})< +\infty$.
Suppose  $M^m$ is a complete Riemannian manifold with a pole and
 its radial curvature satisfies one of the following two conditions:

(i) $-\frac{A}{(1+r^2)^{1+\epsilon}} \leq K_r \leq \frac{B}{(1+r^2)^{1+\epsilon}}$ with $\epsilon>0$, $A \geq 0$, $0 \leq B < 2\epsilon$
and $1 + (m-1)(1-\frac{B}{2\epsilon})-2 d_F e^{\frac{A}{2\epsilon}} \geq m-2$;

(ii) $-\frac{a^2}{1+r^2} \leq K_r \leq \frac{b^2}{1+r^2}$ with $a\geq 0$,  $b^2 \in[0,1/4]$ and
$1 + (m-1)\frac{1+\sqrt{1-4b^2}}{2}- d_F(1+\sqrt{1+4a^2})  \geq (m-1)A^{'}-1$, where $A^{'}=\frac{1+\sqrt{1+4a^2}}{2}$.

If $u(x)\rightarrow p_o \in N^n$ as $r(x)\rightarrow \infty$, then $u$ is a constant map.
 \end{corollary}
\begin{proof}
For the first case (i), it follows that
\begin{equation*}
 Ric_{g_0}(x) \geq -\frac{(m-1)A}{(1+r^2(x))^{1+\epsilon}} \ \ \forall \ x\in M^m.
 \end{equation*}
 By direct calculation
\begin{equation*}
\int^{\infty}_0\frac{Ar}{(1+r^2)^{1+\epsilon}}dr=\frac{A}{2\epsilon},
\end{equation*}
Thus using the volume comparison theorem (cf. [PRS]), we have
\begin{equation*}
vol_{g_0}(\partial B(R)) \leq \omega_me^{\frac{(m-1)A}{2\epsilon}}R^{m-1}
\end{equation*}
where $\omega_m$ is the $(m-1)$-volume of the unit sphere in $R^m$, and thus
\begin{equation*}
(\int_R^{\infty}\frac{dr}{vol_{g_0}(\partial B(r))})^{-1} \leq (m-2)\omega_me^{\frac{(m-1)A}{2\epsilon}}R^{m-2} \ \ for \ \ R>R_0.
\end{equation*}

For the second case (ii), it follows that
 \begin{equation*}
 Ric_{g_0}(x) \geq -\frac{(m-1)a^2}{1+r^2(x)} \ \ \forall \ x\in M^m.
 \end{equation*}
 Then the volume comparison theorem yields (cf. [PRS])
  \begin{equation*}
vol_{g_0}(\partial B(R)) \leq CR^{(m-1)A^{'}}
\end{equation*}
where  $A^{'}=\frac{1+\sqrt{1+4a^2}}{2}$. Thus
\begin{equation*}
(\int_R^{\infty}\frac{dr}{vol_{g_0}(\partial B(r))})^{-1} \leq CR^{(m-1)A{'}-1} \ \ for \ \ R>R_0.
\end{equation*}
Therefore, using Corollary 5.1, the conclusion of this corollary is immediately proved.
\end{proof}

\begin{remark}
If $-\alpha^2\leq K_r \leq -\beta^2$ with  $\alpha> 0$, $\beta> 0$, then the volume of $\partial B(R)$ has exponential growth, thus the condition (5.2) doesn't hold for any $\sigma >0$, so we needn't consider the case (i) in Theorem 5.3.
\end{remark}

\begin{corollary}
 Let $u: (R^m, g_0)\rightarrow (N^n, h)$ be a $C^2$ $F$-harmonic map. Suppose the $F$-lower
degree $l_F>0$ and $F'(\frac{|du|^2}{2})< +\infty$. If $d_F\leq 1$ and $u(x)\rightarrow p_0 \in N^n$ as $|x|\rightarrow \infty$, then $u$ is a constant map.
\end{corollary}

\begin{proof}
 Consider the case $A=B=0$ in Corollary 5.2 (i). If $d_F\leq 1$, then the conditions in that corollary are satisfied by choosing $\sigma=m-2$, and the conclusion of this corollary follows immediately.
\end{proof}

\begin{remark}
For the harmonic map $u$, it is an $F$-harmonic map with $F=2t$, $d_F=l_F=1$,  $F^{'}(t)=2$.  Thus  Theorem A of [Ji]  stated in the introduction can be regarded as one version of this corollary in the case of  harmonic maps.
\end{remark}

\begin{theorem}
Suppose  $f$ satisfies $(f_1)$ and $(f_3)$ and the $F$-lower
degree $l_F>0$. Then for any $p \in N^n$, there is an (nonempty) open neighbourhood $U_p \subset N^n$, such that the family of open sets
 $\{U_p \mid p \in N^n\}$ has following property:

 If $u: (M^m, f^2g_0) \rightarrow (N^n, h)$ is a $C^2$ harmonic map,  $F'(\frac{|du|^2}{2})< +\infty$, and
 for some $p \in N^n$, $u(x) \in U_p$ as $r(x) \rightarrow \infty$, then $u$ is  a constant map.
 \end{theorem}

 \begin{proof}
 The proof of Theorem 5.4 is a modification of the proof of Theorem 5.1. Choose a family of  coordinate neighbourhoods
 $\{U_p \mid p \in N^n\}$ as follows:   let $(U_p, \varphi)$ be a coordinate system centered at  $p$ such that
 \begin{equation*}
 (\frac{\partial h_{\alpha\beta(y)}}{\partial y^\gamma}y^\gamma +2h_{\alpha\beta}(y)) \geq (h_{\alpha\beta}(y)) \qquad on \ U_p
 \end{equation*}
 and $$h_{\alpha\beta}(y)y^\alpha y^\beta \leq C_p$$
 where $C_p$ is an arbitrary constant which may depend on $p$.
 Then we claim that this family $\{U_p \mid p \in N^n\}$ is what we want.

 In fact, if $u: (M^m, f^2g_0) \rightarrow (N^n, h)$ is a non-constant $C^2$ harmonic map, and for some $p \in N^n$, $u(x) \in U_p$
 as $r(x)\rightarrow \infty$, then we may assume that for some $R_0$, $u(x) \in U_p$ for $r(x) > R_0$.  Proceeding as in the proof of Theorem 5.1, we get
 \begin{equation}
 \frac{1}{Z(R)}\geq C \int_R^{\infty}\frac{1}{M(r)}dr,  \ \ for \ R > R_3.
 \end{equation}
 But in this case
 \begin{equation*}
 M(R) \leq C C_p\underset{\partial B(R)}{\int}f^{m-2}(x)  ds_{g_0} \leq CC_p R\log R.
 \end{equation*}
 Therefore
 \begin{equation*}
\int_R^{\infty}\frac{1}{M(r)}dr \geq \frac{1}{CC_p}\int_R^{\infty}\frac{1}{r\log R  }dr=\infty.
\end{equation*}
 Now we have a contradiction to (5.3), since if u is not a constant map, $Z(R)>0$ for $R$ large.
\end{proof}

It is interesting to note that different $F(t)$ may have the same upper degree or the lower degree. Therefore the results in this section may be applied simultaneously  to different $F$-harmonic maps.

\section{A further theorem and its application}

In this section, we show that the asymptotic condition on $F$-harmonic maps for Liouville theorems can be relaxed if the target manifold is more special. Let $(R^n, h_0)$ be  the standard Euclidean space, where $h_0=\underset{\alpha=1}{\overset{n}{\sum}}(dy_{\alpha})^2$.
 Denote by $\rho(y)=dist_{h_0}(y, o)$ the standard Euclidean distance relative to the origin. If we choose a function $\lambda(\rho)= k_1\rho^{k-1}$, where $k_1>0$ and $k\geq 1$, then it is easy to verify that the metric $h(y)=\lambda^2(\rho(y))h_0(y)$ satisfies
 \begin{equation*}
 (\frac{\partial h_{\alpha\beta(y)}}{\partial y^\gamma}y^\gamma +2h_{\alpha\beta}(y)) \geq  (h_{\alpha\beta}(y)).
 \end{equation*}
Let $u: (M^m, g)\rightarrow (R^n, h)$ be a $F$-harmonic map.
Notice that the whole image of $u$ is contained in a global coordinate of $R^n$.
So one may construct the variation $u+ tw$ as in $\S4$. Therefore all integral formulae and inequalities in $\S4$ still hold just by taking $R_1=0$.

Set
\begin{equation*}
K(R)=\underset{r(x)=R}{\max}\{F'(\frac{|du|^2}{2})h_{\alpha\beta}(u)u^\alpha u^\beta\}.
\end{equation*}
By using a small modified method of Theorem 5.1, we can get the following theorem.
\begin{theorem}
Let $u$: $(M^m, f^2g_0)\rightarrow (R^n, h)$ be a  $C^2$ $F$-harmonic map. Suppose that $f$ satisfies $(f_1)$ and $(f_2)$ and  the $F$-lower degree
$l_F>0$. If  $\lim_{R\rightarrow \infty}K(R)=0$, then $u$ is a constant map.
 \end{theorem}
 \begin{proof}
 By (4.7), we have
 \begin{equation*}
 M(R)\leq K(R)\underset{\partial B(R)}{\int} f^{m-2}(x)ds_{g_0}.
\end{equation*}
Since  $K(R)$ is  continuous and $\lim_{R\rightarrow \infty}K(R)=0$, we can choose a function $\overline{K}(R)$ such that\\
\hspace*{0.4cm} (i) $\overline{K}(R)\geq K(R);$\\
\hspace*{0.4cm}(ii) $\overline{K}(R)$ is nonincreasing on $(R_3, \infty)$ and $\overline{K}(R)\rightarrow 0$ as $R\rightarrow \infty$, where $R_3$ is the constant in $\S4$.\\
Thus the remaining part of the proof is similar to that of Theorem 5.1.  We omit the details.
\end{proof}

It follows immediately that
\begin{corollary}
Let $u$: $(M^m, f^2g_0)\rightarrow R$ be a  $C^2$ $F$-harmonic function. Suppose that $f$ satisfies $(f_1)$ and $(f_2)$, and  the $F$-lower degree
$l_F>0$. If
\begin{equation*}
\underset{R\rightarrow \infty}{\lim} \underset{r(x)=R}{\max}\{F'(\frac{|du|^2}{2})|u|^2 \}=0,
\end{equation*}
then $u$ is  a constant map.
 \end{corollary}

\begin{corollary}
Let $x_{m+1}= u(x_1, \cdots, x_m)$ be an entire minimal graph on $R^m$ $(m>2)$, where $x=(x_1, \cdots, x_m) \in R^m$ is the standard Euclidean coordinate. If
\begin{equation*}
\underset{R\rightarrow \infty}{\lim} \underset{|x|=R}{\max}\{\frac{(u-c)^2}{\sqrt{1+|du|^2}}\}=0,
\end{equation*}
then the graph is a horizontal  hyperplane.
\end{corollary}
\begin{proof}
From Example 2.2, we know that  $u: (R^m, g_0)\rightarrow R$ is a $F$-harmonic map with $F(t)=\sqrt{1+2t}-1$.

By a direct calculation, we have $l_F=\frac{1}{2}$, $d_F=1$, $F^{'}(\frac{|du|^2}{2})= \frac{1}{\sqrt{1+|du|^2}} < +\infty$, and $\lambda_{min}=\lambda_{max}=2$. Thus
\begin{equation*}
\frac{(m-1)}{2}\lambda_{\min} + 1
- d_F\max\{2,\lambda_{\max}\}=m-2
\end{equation*}
and
\begin{equation*}
(\int_R^{\infty}\frac{dr}{vol(\partial B(r))})^{-1}=m(m-2)\omega_mR^{m-2}
\end{equation*}
where $\omega_m$ is the volume of the unit sphere in $R^m$. Clearly the minimality of the graph is invariant under the upward or downward movement in $x_{m+1}$-axis direction. Thus the conditions of Corollary 6.1 are satisfied if we
choose $C=m(m-2)\omega_m$ and $\sigma=m-2$,  therefore $u$ is  constant  and the graph is a horizontal  hyperplane.
\end{proof}

\begin{remark}
 (i) When $m \leq 7$, it is well known that the entire graph $x_{m+1}= u(x_1, \cdots, x_m)$ over $R^m$ is a hyperplane in $R^{m+1}$(cf. [Si]).
We know that  $\frac{1}{\sqrt{1+|du|^2}}=\cos(\theta)$, where $\theta$ is the angle function between the normal vector field of the graph and the $x_{m+1}$-axis. Actually we may state  Corollary
6.2 in a more geometric way: Let $P$ be an $m$-dimensional hyperplane in $R^{m+1}$. Suppose
$M^m$ is a complete minimal hypersurface in $R^{m+1}$ which is a graph over the plane
$P$. Denote by $\theta$ the angle function between the normal vector field of $M$ and the
normal vector of $P$. For $q\in M$, let $d(q, P)$ be the Euclidean distance from $q$ to
$P$. If
\begin{equation*}
\underset{q\rightarrow \infty}{\lim}d^2(q, P)\cos\theta(q)=0,
\end{equation*}
then $M$ is a hyperplane parallel to $P$.

(ii) From Corollary 6.2, we may deduce the following result: If there exists a constant $c$ such that
  $\underset{|x|\rightarrow \infty}{\lim} u(x)=c$,  then $u=c$.
  Note that the later result is also a consequence of Simon's result (cf. Lemma 1.1 in [Sim]).
\end{remark}

\vspace{0.5cm}
$\mathbf{Acknowledgements}$. The authors would like to thank Prof. S.S. Wei and Prof. T.H. Otway for their valuable comments  and suggestions. They would also like to thank Dr. Qingchun Ji  and Dr. Y.B. Ren for their helpful discussions.

%\end{CJK*}
\end{document}